\documentclass[12pt]{article}
\usepackage{amsmath,amssymb,amsthm}
\usepackage{amsfonts, amsmath}
\newtheorem{theorem}{Theorem}[section]
\newtheorem{lemma}{Lemma}[section]
\newtheorem{corollary}{Corollary}[section]

\theoremstyle{definition}
\newtheorem{definition}{Definition}[section]
\theoremstyle{remark}
\newtheorem{remark}{Remarks}[section]
\newtheorem{example}{Example}[section]
\numberwithin{equation}{section}

 \title{A mathematical theory of truth and an application to the regress problem} 
 
\author {S. Heikkil\"a\\
Department of Mathematical Sciences, University of Oulu\\
BOX 3000, FIN-90014, Oulu, Finland\\
E-mail: sheikki@cc.oulu.fi}

\begin{document}
\maketitle 
Shorttitle: A mathematical theory of truth and an application

\noindent
\begin{abstract} 
\noindent
In this paper a class of languages  which are formal enough for mathematical reasoning is introduced. First-order formal languages containing natural numbers and numerals belong to that class. Its languages are called
mathematically  agreeable (shortly MA). Languages containing a given MA language $L$, and being sublanguages of $L$ augmented by a monadic predicate, are constructed. A mathematical theory of truth (shortly MTT) is formulated for some of these languages. MTT makes them MA languages which posses their own truth predicates. MTT is shown to conform well with the  eight norms presented for theories of truth in the paper `What Theories of Truth Should be Like (but Cannot be)', by Hannes Leitgeb. 
MTT is also free from infinite regress, providing a proper framework to study the regress problem. Main tools used in proofs are Zermelo-Fraenkel (ZF) set theory and classical logic.  

\vskip12pt

\vskip12pt


\end{abstract}
\newpage

\baselineskip 15pt


\section {Introduction}\label{S1} 
In this paper 
a theory of truth is formulated for a class of languages. 
The regress problem is studied within the framework of that theory.

\smallskip
A language $L$ is called mathematically agreeable (shortly MA), if it satisfies the following three conditions.
\vskip2pt

(i) $L$ contains a countable syntax of the first-order predicate logic with equality (cf., e.g.,  \cite[Definitions II.5.1--5.2.6]{[Ku]}), natural numbers in variables and their names, numerals in terms.

 
(ii) $L$ is fully interpreted, i.e., every sentence of $L$ is interpreted either as true or as false.  

(iii) Classical truth tables (cf. e.g., \cite {[Ku]}, p.3) are valid for the logical connectives $\neg$, $\vee$, $\wedge$, $\rightarrow$ and
$\leftrightarrow$ of sentences of $L$, and classical rules of truth hold for applications of quantifiers
$\forall$ and $\exists$ to formulas of $L$. 
\smallskip

These properties ensure that every MA language is formal enough for mathematical reasoning.

 Any countable first-order formal language, equipped with a consistent theory interpreted by a countable model, and containing natural numbers and numerals, is an MA language. A classical example is the language of arithmetic  with its standard model and interpretation. 

Basic ingredients of the present approach are: 

1. An MA language $L$ (base language). 

2. A monadic predicate $T$  having the set $X$ of numerals as its domain of discourse.

3. The language $\mathcal L$, which has  sentences of $L$, $T(\mathbf n)$, $\mathbf n\in X$, $\forall xT(x)$ and $\exists xT(x)$ as its basic sentences, and which is closed under  logical connectives $\neg$, $\vee$, $\wedge$, $\rightarrow$ and
$\leftrightarrow$.

4. The set $D$ of G\"odel numbers of sentences of $\mathcal L$ in its fixed G\"odel numbering. 

The paper is organized as follows. 

In Section 2 we construct to each subset $U$ of $D$ new subsets  $G(U)$ and $F(U)$ of $D$. Let
$\mathcal L_U$ be the language of the sentences whose G\"odel numbers are in $G(U)\cup F(U)$.   It contains $L$.

In Section 3 results on the existence and construction of consistent fixed points of $G$, i.e. consistent sets satisfying $U=G(U)$, including the smallest one, are presented. $U$ is called consistent if for no sentence $A$ of $\mathcal L$ the G\"odel numbers of both $A$ and $\neg A$ are in $U$.

In Section 4 a mathematical theory of truth (shortly MTT) is defined for  languages $\mathcal L_U$, where $U$ is a consistent  fixed point of $G$. A sentence $A$ of $\mathcal L_U$ is interpreted as true if its G\"odel number \#$A$ is in $G(U)$, and as false if \#$A$ is in $F(U)$. This makes $\mathcal L_U$ an MA language. 
$T$ is called a truth predicate for $\mathcal L_U$. Biconditionality: 
$A\leftrightarrow T(\left\lceil A\right\rceil)$, where $\left\lceil A\right\rceil$ is the numeral of the G\"odel number of $A$,  is shown to be true for all sentences $A$ of $\mathcal L_U$.
Since both $L$ and $\mathcal L_U$ are fully interpreted, their sentences are either true or false. 
Moreover, a sentence $A$ of $L$ is either true or false in the interpretation of $L$ if and only if $A$ 
is either true or false in $\mathcal L_U$.

Section 5 is devoted to the study of the regress problem  within the framework of MTT. 
We present an example of an infinite regress (parade) of justifications
that satisfies the  conditions imposed on them in \cite{[17]}.  
Example is inconsistent with the following conclusion stated in \cite{[18]}: 
``it is logically impossible for there to be an infinite parade of justifications''. 
That conclusion is used in \cite{[17],[18]} as a basic argument to refute Principles of Sufficient Reasons. 

In Section 6 we shall first introduce some benefits of MTT compared with some other theories of truth. 
The lack of Liar-like sentences makes  MTT mathematically acceptable.
 MTT is shown to conform well with  the eight norms presented in \cite{[16]} for theories of truth.   
 Connections of obtained results to  mathematical philosophy and epistemology are also presented. 
 

\section{Construction of languages}\label{S2} 

Let basic ingredients  $L$, $T$,  $\mathcal L$ and $D$ be 
as in the Introduction.
We shall  construct a family of sublanguages for the language $\mathcal L$. As for the used terminology, cf. e.g.,  \cite{[Ku]}.
Let $U$  be a subset of  $D$. 
Define subsets  $G(U)$ and $F(U)$ of $D$ by following rules, which are similar to those presented in \cite{[11]} ('iff' abbreviates 'if and only if'):  
 \begin{enumerate}
 \item[(r1)] If $A$ is a sentence of $L$, then the G\"odel number \#$A$ of $A$ is in  $G(U)$ iff  $A$ is true in the interpretation of $L$, and  in  $F(U)$ iff $A$ is false in the interpretation of $L$.  
 \item[(r2)]  Let $\mathbf n$ be a numeral. $T(\mathbf n)$ is in $G(U)$ iff $\mathbf n=\left\lceil A\right\rceil$, where $A$ is a sentence of $\mathcal L$ and  \#$A$  is in $U$. $T(\mathbf n)$ is in $F(U)$ iff $\mathbf n=\left\lceil A\right\rceil$, where $A$ is a sentence of $\mathcal L$ and \#[$\neg A$] is in $U$. 
\end{enumerate}  
Sentences determined by rules (r1) and (r2), i.e., all sentences $A$ of $L$ and those sentences $T(\left\lceil A\right\rceil)$ of $\mathcal L$ for which \#$A$ or \#[$\neg A$]  is in $U$, are called {\em basic sentences}. 

Next rules deal with logical connectives. Let $A$ and $B$ be sentences of $\mathcal L$. 
\begin{enumerate}
 \item[(r3)] Negation rule: \#[$\neg A$] is in $G(U)$ iff \#$A$ is in $F(U)$, and in $F(U)$ iff \#$A$ is in $G(U)$. 
 \item[(r4)] Disjunction rule: \#[$A\vee B$]
 is  in $G(U)$ iff  \#$A$ or \#$B$ is in $G(U)$, and in $F(U)$ iff \#$A$ and \#$B$
are in $F(U)$. 
 \item[(r5)] Conjunction rule: \#[$A\wedge B$]
is in $G(U)$ iff \#[$\neg A\vee \neg B$] is in $F(U)$ iff (by (r3) and (r4))  both \#$A$ and \#$B$ are in $G(U)$. Similarly, \#[$A\wedge B$] is in $F(U)$ iff \#[$\neg A\vee \neg B$] is in $G(U)$ iff  \#$A$ or \#$B$ is in $F(U)$.
\item[(r6)] Implication rule: \#[$A\rightarrow B$] is in $G(U)$ iff \#[$\neg A\vee B$] is in $G(U)$ iff (by (r3) and (r4))  \#$A$ is in  $F(U)$ or \#$B$ is in $G(U)$.  \#[$A\rightarrow B$] is  in $F(U)$ iff \#[$\neg A\vee B$] is in $F(U)$ iff \#$A$ is in $G(U)$ and \#$B$ is in $F(U)$.
\item[(r7)] Biconditionality rule:
\#[$A \leftrightarrow B$] is in $G(U)$ iff \#$A$ and \#$B$ are both in $G(U)$ or both in $F(U)$, and  in $F(U)$ iff  \#$A$ is in $G(U)$ and 
\#$B$ is in $F(U)$ or \#$A$ is in $F(U)$ and \#$B$ is in $G(U)$.
\end{enumerate}
Rule (r1) is applicable for sentences of $L$ formed by applications of universal and existential quantifiers to formulas of $L$.
Thus it suffices to set rules for $\exists xT(x)$ and $\forall xT(x)$. 
\begin{enumerate}
\item[(r8)] \#[$\exists xT(x)$] is in $G(U)$ iff \#$T(\mathbf n)$ is in $G(U)$ for some numeral $\mathbf n$. \#[$\exists xT(x)$] is in $F(U)$ iff  \#$T(\mathbf n)$ is in $F(U)$ for every numeral $\mathbf n$.
\item[(r9)] \#[$\forall xT(x)$] is in $G(U)$ iff \#$T(\mathbf n)$ is in $G(U)$ for every numeral $\mathbf n$, and \#[$\forall xT(x)$] is in $F(U)$ iff \#$T(\mathbf n)$ is in $F(U)$ at least for one numeral $\mathbf n$.
\end{enumerate}
Rules (r0)--(r9) and induction on the complexity of formulas
determine uniquely subsets  $G(U)$ and $F(U)$ of $D$ whenever $U$ is a subset of $D$. 
Denote by $\mathcal L_U$ the language formed by all those sentences $\mathcal L$ whose  G\"odel numbers are in $G(U)$ or in $F(U)$. $\mathcal L_U$  contains by rule (r1) all sentences of the base language $L$. 

\section{Fixed point results}\label{S3} 

We say that a subset $U$ of  $D$ is  {\em consistent} if for no sentence $A$ of $\mathcal L$ the G\"odel numbers of both $A$ and $\neg A$ are in $U$. For instance, the empty set $\emptyset$ is consistent. Let $\mathcal P$ denote the family of all consistent subsets of the set $D$ of
G\"odel numbers of sentences of $\mathcal L$.

The following three lemmas can be proved  as the corresponding results in \cite{[11]}, replacing  'true in $M$' by 'true in the interpretation of $L$'.

\begin{lemma}\label{L201} (\cite[Lemma 2.1]{[11]})
 If $U\in\mathcal P$, then $G(U)\in\mathcal P$, $F(U)\in\mathcal P$, and $G(U)\cap F(U)=\emptyset$.
\end{lemma}

According to Lemma \ref{L201} the mapping $G:=U\mapsto G(U)$ maps $\mathcal P$ into $\mathcal P$.
Assuming  that $\mathcal P$ is ordered by inclusion, we have the following result.

\begin{lemma}\label{L203} (\cite[Lemma 4.2]{[11]}) $G$ is order preserving in $\mathcal P$, i.e.,
 $G(U)\subseteq G(V)$ whenever  $U$ and $V$ are  sets of  $\mathcal P$ and $U\subseteq V$.
\end{lemma}

\begin{lemma}\label{L204} (\cite[Lemma 4.3]{[11]})
If $\mathcal W$ is a chain in $\mathcal P$, then 
$\cup \mathcal W=\cup\{U\mid U\in \mathcal W\}$ is in $\mathcal P$. 
\end{lemma}

 Fixed points of  the set mapping $G:=U\mapsto G(U)$ from $\mathcal P$ to $\mathcal P$, i.e., those $U\in \mathcal P$ for which $U=G(U)$, have a central role in the formulation of MTT. 
In the formulation our main fixed point theorem we use transfinite sequences of $\mathcal P$ indexed by von Neumann ordinals. Such a sequence $(U_\lambda)_{\lambda\in\alpha}$ of $\mathcal P$ is said to be strictly increasing if
$U_\mu\subset U_\nu$ whenever $\mu\in\nu\in\alpha$.
A set $V$ of $\mathcal P$ is called  {\em sound} iff $V\subseteq G(V)$.

The following fixed point theorem is proved in \cite{[11]}.

\begin{theorem}\label{T1} (\cite[Theorem 4.1]{[11]})
If $V\in \mathcal P$ is  sound, then there exists the smallest of those consistent fixed points of $G$  which contain $V$. This fixed point is the last member of the union of those transfinite sequences
$(U_\lambda)_{\lambda\in\alpha}$ of $\mathcal P$ which satisfy
\begin{itemize}
\item[(C)] $(U_\lambda)_{\lambda\in\alpha}$ is strictly increasing,
$U_0=V$, and if $0\in\mu\in \alpha$, then
$U_\mu = \underset{\lambda\in\mu}{\bigcup}G(U_\lambda)$.
\end{itemize}
\end{theorem}

As a consequence of Theorem \ref{T1} we obtain.
\begin{corollary}\label{C2} Let $W$ be the set of  G\"odel numbers of all those sentences of $L$ which are true in its interpretation,
and let $V$ be any subset of $W$.

(a) $V$ is a  sound and consistent subset of $D$.

(b) The union of the transfinite sequences which satisfy (C) is the smallest consistent fixed point of $G$.
\end{corollary}

\begin{proof} (a) Rule (r1) and Lemma \ref{L203} imply that $V\subseteq G(\emptyset)\subseteq G(V)$, so that $V$ is sound. It is also consistent, as a subset of a consistent set $G(\emptyset)$. 
\smallskip

(b) $V$ is by (a) sound and consistent. If $U$ is a consistent fixed point of $G$, then $V\subseteq G(\emptyset)\subset G(U)=U$. Thus $V$ is contained in every consistent fixed point of $G$. By Theorem \ref{T1}, the union of  those transfinite sequences
$(U_\lambda)_{\lambda\in\alpha}$ of $\mathcal P$ which satisfy (C) is the smallest consistent fixed point of $G$ that contains $V$. This proves (b).
\end{proof}

\begin{remark}\label{R51}
The smallest members of  $(U_\lambda)_{\lambda\in\alpha}$ satisfying (C) are $n$-fold iterations $U_n=G^n(V)$, $n\in\mathbb N=\{0,1,\dots\}$. If they form a strictly increasing sequence, the next member $U_\omega$ is their union, $U_{\omega+n}= G^n(U_\omega)$, $n\in\mathbb N$, and so on.
\end{remark}

\section{A mathematical theory of truth}\label{S4}

Recall that  
$D$ denotes the set of G\"odel numbers of sentences of the language $\mathcal L$. Given  a subset $U$ of  $D$, let $G(U)$ and $F(U)$ be the subsets  of $D$ constructed in Section \ref{S3}. In the next definition, which is the same as presented in \cite{[11]} in a special case, we formulate our mathematical theory of truth (shortly MTT).
  
\begin{definition}\label{D41} {\it Assume that $U$ is a consistent subset of $D$, and that $U=G(U)$.  
Denote by $\mathcal L_U$ the language containing those sentences $A$ of $\mathcal L$ for which   \#$A$ is in $G(U)$ or in $F(U)$.
A  sentence $A$ of $\mathcal L_U$ is interpreted as true iff \#$A$ is in $G(U)$, and as false iff \#$A$ is in $F(U)$. $T$ is called a truth predicate for $\mathcal L_U$}.  \end{definition}

In view of Definition \ref{D41}, `$\#A$ is in $G(U)$' can be replaced by `$A$ is true' and `$\#A$ is in $F(U)$' by 
`$A$ is false' in (r1)--(r9). This replacement, the construction of $G(U)$ and $F(U)$ and Lemma \ref{L201} imply that $\mathcal L_U$ is an MA language, having thus those syntactical and semantical properties which are assumed for the base language $L$.

The following result justifies to call $T$ as a truth predicate of $\mathcal L_U$. 

\begin{lemma}\label{L4.1} If $U$ is a consistent subset of $D$, and  if $U=G(U)$, then $T$-biconditionality: $A\leftrightarrow T(\left\lceil A\right\rceil)$ is true, and $A\leftrightarrow \neg T(\left\lceil A\right\rceil)$ is false for every sentence $A$ of $\mathcal L_U$.
\end{lemma}

\begin{proof} Assume that $U\subset D$ is consistent, and that $U=G(U)$. Let  $A$ be a sentence of $\mathcal L_U$.
Applying rules (r2) and (r3), and the assumption  $U=G(U)$, we obtain
\newline
-- \#$A$ is in $G(U)$ iff \#$A$ is in $U$  iff 
\#$T(\left\lceil A\right\rceil)$ is in $G(U)$ iff \#$\neg T(\left\lceil A\right\rceil)$ is in $F(U)$;
\newline
-- \#$A$ is in $F(U)$ iff \#[$\neg A$] is in $G(U)$ iff \#[$\neg A$] is in $U$  iff \#$T(\left\lceil A\right\rceil)$ is in $F(U)$
iff \#$\neg T(\left\lceil A\right\rceil)$ is in $G(U)$.  
\newline
The above results, rule (r7) and Definition  \ref{D41} imply that $A\leftrightarrow T(\left\lceil A\right\rceil)$ is true, and that $A\leftrightarrow \neg T(\left\lceil A\right\rceil)$ is false. This holds for every sentence $A$ of $\mathcal L_U$.
\end{proof}

Our main result on the connection between the valuations determined by the interpretation of $L$ and that of $\mathcal L_U$ defined in Definition \ref{D41} reads as follows:  

\begin{lemma}\label{L4.2} Let $U$ be a consistent fixed point of $G$. If $A$ is a sentence of  $L$, then either 
\newline
(a) $A$ is true in the interpretation of $L$, iff $A$ is  true,  iff $T(\left\lceil A\right\rceil)$ is true, or 
\newline
(b) $A$ is false in the interpretation of $L$,  iff  $A$ is false,  iff $T(\left\lceil  A\right\rceil)$ is false. 
\end{lemma}

\begin{proof} Assume that $A$ is a sentence of $L$. Because $L$ is completely interpreted, then $A$ is either true or false in the interpretation of $L$. 

-- $A$ is true in the interpretation of $L$ iff \#$A$ is in $G(U)$, by rule (r1), iff \#$A$ is in $U$, because $U=G(U)$, iff  \#$T(\left\lceil A\right\rceil)$ is in $G(U)$ by rule (r2), iff $T(\left\lceil A\right\rceil)$ is true,  by Definition \ref{D41}.

-- $A$ is false in the interpretation of $L$ iff $\neg A$ is true in the interpretation of $L$   iff  \#[$\neg A$] is in $G(U)$, by rule (r1), iff \#[$\neg A$] is in $U$, because $U=G(U)$,  iff  \#$T(\left\lceil A\right\rceil)$ is in $F(U)$, by rule (r2), iff $T(\left\lceil A\right\rceil)$ is false, by Definition \ref{D41}.   
   
Consequently, a sentence $A$ of $L$ is true in the interpretation of $L$ iff $T(\left\lceil A\right\rceil)$ is true, and false in the interpretation of $L$ iff $T(\left\lceil A\right\rceil)$ is false. These results and the result of Lemma \ref{L4.1} imply the conclusions (a) and (b).  
\end{proof}







\section{On the Regress Problem}\label{S6}
\setcounter{equation}{0}

First of ten theses presented in \cite[p. 6]{[2]} is: ``The Regress Problem is a real problem for epistemology."
We are going to study the regress problem in the framework of MTT. 
We adjust first our terminology to that used in \cite{[17]} in the study of the regress problem. Given an MA language $L$, let  an MA language $\mathcal L_U$ that contains $L$ be determined by Definition \ref{D41}, $U$ being the smallest fixed point of $G$. 
By statements we mean the sentences of $\mathcal L_U$, which are valued by Definition \ref{D41}. 
A statement $A$ is said to entail $B$, if it is not possible that $A$ is true and $B$ is false simultaneously.  For instance, if  $A\rightarrow B$ is true, then  $A$ entails $B$. We say that a statement $A$ justifies a statement $B$ if $A$ confirms the truth of $B$. For instance, if 
$A \leftrightarrow \neg B$ is true, then $A$ justifies $B$ iff $A$ is false. If
 $A\rightarrow B$ is true, then $A$ justifies $B$ iff $A$ is true (Modus Ponens). 
$A$ is called contingent if the truth value of $A$ is unknown.
\smallskip

Consider an infinite regress
\begin{equation}\label{E0}
\dots F_i,\dots,F_1,F_0
\end{equation}
 of statements $F_i$, $i\ge 0$,
 where the statement $F_0$ is contingent.
We shall impose the following conditions on  statements $F_i$, $i>0$ (cf. \cite{[17]}): 
\smallskip

\begin{enumerate}
\item[(i)] $F_i$  entails $F_{i-1}$;
\item[(ii)] $F_0\vee\cdots\vee F_{i-1}$ does not entail $F_i$;
\item[(iii)]  $F_0\vee\cdots\vee F_{i-1}$ does not justify $F_i$.
\item[]  Regress (\ref{E0}) is called justification-saturated if the following condition holds: 
\item[(iv)] $\dots$ what justifies $F_{i-1}$ is $F_i,\dots$, what justifies $F_1$ is $F_2$, what justifies $F_0$ is $F_1$. 
\end{enumerate}

\begin{lemma}\label{L10} Assume that in regress (\ref{E0}) the statement $F_0$ is contingent, and that the statements $F_i$, $i>0$, 
satisfy conditions (i)--(iii). 

(a) If  $F_1$ is false, then $F_{i}$ is false for each $i> 0$. $F_0$ is justified iff $F_0\leftrightarrow\neg F_1$ is true.

(b) If $F_{n}$ is true for some $n>0$, then  $F_i$ is true when $0\le i\le n$. 

(c) The regress (\ref{E0}) is justification saturated iff $F_{i}$ is true for all $i > 0$, in which case
 $F_0$ is justified.
\end{lemma}

\begin{proof}  (a) Assume that  $F_1$ is false. If $F_i$ would be true for some $i>1$, there would be the smallest such an $i$. Then $F_{i-1}$ would be true  by (i). Replacing $i$ by $i-1$, and so on, this reasoning would imply after $i-1$ steps that $F_1$ is true;  a contradiction. Thus all  statements $F_i$, $i> 0$, are false. Because $F_1$ is false, it confirms the truth of $F_0$ iff $F_0$ and $\neg F_1$ have same truth values iff
$F_0\leftrightarrow\neg F_1$ is true.

(b)  Assume that  $F_{n}$ is true for some $n>0$.  Since  $F_i$ entails $F_{i-1}$, $i=n,n-1,\dots,1$, then $F_{i}$ is true for every $i=n-1,\dots,0$. 
 
(c) If $F_n$ is false for some $n> 0$, then  $F_{n+1}$ is false by property (i), and it does not justify $F_n$, so that condition (iv) is not valid. On the other hand, condition  (i) ensures that condition (iv) is valid if $F_{i}$ is true for all $i > 0$. In this case $F_1$ justifies $F_{0}$, i.e., $F_0$ is true. 
\end{proof} 


\smallskip

\begin{example}\label{Ex2} Let $L$ be the first-order language $L=\{\in\}$ of set theory, and 
$M$ the minimal model of ZF set theory constructed in \cite{[4]}. $M$ is countable and contains the set $\omega$ of natural numbers and their set $S(\omega)=\omega\cup\{\omega\}$ (cf. \cite{[4],[13]}).
 We assume that numerals are  defined in $L$, e.g., as in  \cite{[Fi]}. 
Interpret a sentence $A$ of $L$ as true in $L$
if  $M{|\!\!\!=} A$, and false in $L$ if $M{|\!\!\!=} \neg A$, in the sense defined in  \cite[II.2.7 and p. 237]{[Ku]}. 
By \cite[Lemma II.2.8.22]{[Ku]} this interpretation makes $L$ fully interpreted.  In particular, $L$ is an MA language. 
Choose $L$ as the base language of theory MTT. 
 Equip $S(\omega)$ with the natural ordering $<$ of natural numbers plus $n<\omega$ for every natural number $n$. 
If $Z$ denotes a nonempty subset of $S(\omega)$, it is easy to verify that the infinite regress (\ref{E0}) of statements
\begin{equation}\label{E3}
F_i: \quad i<\beta, \ \mbox{for every } \ \beta\in Z, \quad i=0,1,\dots,
\end{equation}
satisfy conditions (i)--(iii), and that $F_0$ is contingent. 
Moreover, condition (iv) is valid by Lemma \ref{L10} if and only if 
$F_{i}$ is true for all $i > 0$. This holds   if and only if $Z=\{\omega\}$. 
\end{example} 
\smallskip


\section{Remarks}\label{S8}
The main purpose of this paper is to present a mathematical theory of truth (MTT) for a class languages which are formal enough for mathematical reasoning.  

To describe properties of MTT and to compare it to some other theories of truth, let $S=(L,\Sigma)$ be a mathematical theory, where $L$ is a first-order formal language, and $\Sigma$ is a set of axioms. 
Assume that $\Sigma$ is consistent, and is either an extension of Robinson arithmetic $Q$ (e.g., $Q$ itself or Peano arithmetic, $L$ being the language of arithmetic), or $Q$ can be interpreted in $\Sigma$ (e.g., $\Sigma$ axiomatizes ZF set theory, and $L$ is the language of set theory). 
By L\"owenheim-Skolem theorem that theory has a countable model $M$. Interpret a sentence $A$ of $L$ as true in $L$
if  $M{|\!\!\!=} A$, and false in $L$ if $M{|\!\!\!=} \neg A$, in the sense defined in  \cite[II.2.7]{[Ku]}. 
By \cite[Lemma II.2.8.22]{[Ku]} this interpretation makes $L$ fully interpreted, and $L$ is an MA language. 
Tarski's Undefinability Theorem  (cf. \cite{[20]}) implies that $L$ cannot contain it's truth predicate, yielding `Tarski's Commandment' (cf.  \cite{[MG]}). Let $\mathcal L$ be a formal language obtained by augmenting $L$ with a monadic predicate $T$. 
$T$ cannot be a truth predicate of $\mathcal L$, for otherwise one could construct a Liar sentence, which implies  the `Liar paradox' (cf. \cite{Ho}). Thus $\mathcal L$ does not contain its truth predicate, either. Many axiomatic theories of truth (cf., e.g, \cite{Fe}) are constructed for languages which contain a Liar sentence or are subject to the 'Revenge of Liar'. Such languages  are not MA languages.     
\vskip4pt
 
Theory MTT provides an alternative. Given an MA language $L$, let $\mathcal L_U$, where $U$ is a consistent fixed point of $G$, 
be an extension of $L$ constructed in Section \ref{S3}. That construction and the interpretation given for $\mathcal L_U$ in Definition \ref{D41} makes it an MA language. Moreover, $\mathcal L_U$ contains by Definition \ref{D41} a truth predicate $T$.
It follows from Lemma \ref{L4.1} that  there is no Liar sentence in $\mathcal L_U$.  As an MA language $\mathcal L_U$ {\em is formal enough for mathematical reasoning}. 
In particular, the language $L$ of the above theory $S$
 is extended  in S$_U=$($\mathcal L_U$,MTT) to an MA language $\mathcal L_U$ that contains its truth predicate and is free from paradoxes.
\vskip4pt

If $S=(L,\Sigma)$ is as above,
 $L$ contains by G\"odel's First Incompleteness Theorem a true arithmetical sentence, say $B$, that is not provable from the axioms of $\Sigma$ (cf. \cite{R}).  By Lemma \ref{L4.2} both $B$ and $T(B)$ are true in the interpretation of $\mathcal L_U$. Based on the existence of $B$  the following opinions on mathematical truth 
presented in \cite[Chapter 4]{[Pe]}: ``The notion of mathematical truth goes beyond the whole concept of formalism. There is something absolute and `God-given' about mathematical truth. 
Real mathematical truth goes beyond mere man-made constructions."  These opinions are questioned because of the assumption that $\Sigma$ is  consistent. (cf. \cite{R}).  Despite inability of human mind to see that consistency it is indispensable for reliability mathematical results. Mathematics rests on the belief that its theories are consistent.
\vskip4pt

MTT has properties that conform well with the eight norms formulated in \cite{[16]} for theories of truth. 
Truth is expressed by a predicate $T$. An MA language $\mathcal L_U$ contains a syntax of first-order logic with equality, natural numbers as constants and numerals as terms. It is closed under  logical connectives and quantifiers.
A theory of truth is added to the base language $L$. If the interpretation of $L$ is determined by a consistent mathematical theory (Peano arithmetic, ZF set theory, e.t.c.), then  MTT proves the theory in question true, by Lemma \ref{L4.2}.
 Truth predicate $T$ is not subject to any restrictions within a fixed point language $\mathcal L_U$.
$T$-biconditional is derivable unrestrictedly  within a fixed point language $\mathcal L_U$, by the proof of Lemma \ref{L4.1}. 
Truth is compositional, by Definition \ref{D41} and rules (r3)--(r9).
The theory allows for standard interpretations if the interpretation of $L$ is standard. 
In particular, {\em the outer logic and the inner logic coincide, and they  are classical}. 

Paradoxes led Zermelo to axiomatize set theory.
To avoid paradoxes Tarski ``excluded all Liar-like sentences from being
well-formed", as noticed in \cite{[16]}. 
A fixed point language $\mathcal L_U$ does not contain such sentences in the theory MTT. In particular, MTT is  immune to `Tarski's Commandment' (cf.  \cite{[MG]}), to  Tarski's Undefinability Theorem  (cf. \cite{[20]}), to `Tarskian hierarchies' (cf. \cite{[H]}), and to `Liar paradox' (cf. \cite{Ho}). 
The smallest of those languages  for which MTT is formulated is $\mathcal L_U$, where $U$ is the smallest consistent fixed point of $G$. It relates to that of the grounded sentences defined in \cite{[10],[15]} when $L$ is the language of arithmetic. See also \cite{[6]}, where  considerations are restricted to signed statements.

 A base language $L$ can contain more sentences than first-order formal languages, thus extending the class of languages for which theories of truth are usually formulated.  
\smallskip

Another purpose of the presented theory of truth  is to establish a proper framework to study  the regress problem. 
Tarski's theory of truth (cf. \cite{[20]}) does not offer it because  that theory itself is not free from infinite regress.
According to \cite[p.189]{[P]}: ``the most
important problem with a Tarskian truth predicate is its demand
for a hierarchy of languages. ...
within that hierarchy of languages, we cannot
seem to have any valid method of ending the regression to
introduce the ``basic" metalanguage."

\smallskip  Kripke's theory of truth is also a problematic framework because of three-valued inner logic.
 As stated in \cite[p.283]{[16]}: ``Classical first-order logic is certainly the default choice
for any selection among logical systems. It is presupposed by standard
mathematics, by (at least) huge parts of science, and by much of philosophical
reasoning." 
Moreover, $T$-biconditionality rule does not hold in Kripke's theory of truth because of paradoxical sentences.

Example \ref{Ex2} is inconsistent with the conclusion of \cite{[18]} cited in the Introduction.
In this example the property that regress (\ref{E0}) is justification-saturated both implies and is implied by truth of a 'foundational' statement $F_b: Z=\{\omega\}$. Thus  it  does not support the form of infinitism presented in \cite{[14]}: ``infinitism holds that there are no ultimate, foundational reasons". Pure infinite regress is even refused in \cite[p.13]{[F]}.
On the other hand, 
it supports ``impure" infinitism and the form of foundationalism presented in \cite{[2],[21]}.
 
Example \ref{Ex2} implies that the proofs in \cite{[17],[18]} to the assertion that ``any version of Principle of Sufficient Reason is false" are based on the questionable premise  that infinite regresses of justifications  don't exist.  
In fact, this  example gives some support to Principles of Sufficient Reason, as well as to many other arguments 
whose validity is questioned in \cite{[17],[18]}.  
For instance, in the `universe' $S(\omega)$ of Example \ref{Ex2},
\begin{itemize}
\item $\{\omega\}$ provides a {\em sufficient reason} for  $F_0$; 
\item $\{\omega\}$ affords an {\em ultimate and foundational reason} that justifies $F_0$;
\item $\{\omega\}$ is the {\em final explainer} of $F_0$;
\item $\{\omega\}$ gives the {\em first cause} that makes regress (\ref{E0}),(\ref{E3}) justification-saturated; 
\item $\omega$ {\em explains the existence of the 'universe'} $\mathbb N$ of natural numbers  ($\mathbb N=\omega$ by \cite{[12]});
\item $\omega$ and $\{\omega\}$ {\em explain the existence of the `universe'}  $S(\omega)$ ($S(\omega)=\omega\cup\{\omega\}$ by \cite{[12]});
\item $\omega$ is {\em something beyond  natural numbers};
\item $\omega$ is {\em infinite and  greatest} in the `universe' $S(\omega)$;
\item $\omega$ is {\em `self-justified'} (The Axiom of Infinity). 
\end{itemize}

Belief that $\omega$ exists is a matter of  faith.
In Example  \ref{Ex2} we have assumed it because the model $M$ of ZF set theory contains the set $\omega\cup \{\omega\}$. 
 Notice that this set does not belong to the standard model of arithmetic. Thus MTT, where the base language $L$ is the language of arithmetic, is not a sufficient framework for Example \ref{Ex2}. 

 
\vskip16pt

{\bf Acknowledgments:} The author is indebted to Ph.d. Markus Pantsar for valuable discussions on the subject.

\end{document}